\documentclass[11pt, a4paper]{article}

\usepackage{amsfonts,enumerate,amsmath,amsfonts,xspace,array,delarray,theorem,amssymb,epsfig}
\usepackage{graphicx,color}
\usepackage[latin1]{inputenc}%

\theoremstyle{break}\theorembodyfont{\it}
\theoremheaderfont{\normalfont\bfseries}

\newtheorem{theo}{Theorem}

\newtheorem{defi}[theo]{Definition}

\newtheorem{lem}[theo]{Lemma}

\newtheorem{prop}[theo]{Proposition}

\newtheorem{rem}[theo]{Remark}

\newenvironment{proof}{\noindent{\bf Proof: }}
                {\leavevmode\unskip\nobreak\hskip2em plus1fill
                $\scriptstyle\square$\vskip\theorempostskipamount\par}

\def\dis{\displaystyle}

\let\lt=<
\let\gt=>
\def\freccia{{\longrightarrow}}
\def\F { \mathcal{F}}
\def\R{{\mathbb R}}

\let\phi=\varphi
\let\eps=\varepsilon

\def\be{\begin{equation}}
\def\ee{\end{equation}}
\def\beq{\begin{eqnarray*}}
\def\eeq{\end{eqnarray*}}
\def\e{{\rm e}}
\def\d{{\rm d}}
\newcommand{\fer}[1]{(\ref{#1})}

\def\ff{\widehat f}

\def\derpar#1#2{\frac{\partial#1}{\partial#2}}

\def\supp{\mathop {supp}}
\begin{document}

\title{The grazing collision limit of the inelastic \\ Kac model around a L\'evy-type equilibrium.
}

\author{G. Furioli \thanks{University of Bergamo, viale
Marconi 5, 24044 Dalmine, Italy. \texttt{giulia.furioli@unibg.it}},
 A. Pulvirenti \thanks{Department\
of Mathematics, University of Pavia, via Ferrata 1, 27100 Pavia,
Italy. \texttt{ada.pulvirenti@unipv.it}}, E. Terraneo
\thanks{Department of Mathematics, University of Milano, via Saldini
50, 20133 Milano, Italy. \texttt{Elide.Terraneo@mat.unimi.it}}, and
G. Toscani\thanks{Department of Mathematics, University of Pavia,
via Ferrata 1, 27100 Pavia, Italy.
\texttt{giuseppe.toscani@unipv.it}}
 }

\maketitle

\begin{center}\small
\parbox{0.85\textwidth}{
\textbf{Abstract.} This paper is devoted to the grazing collision
limit of the inelastic Kac model introduced in \cite{PT04}, when the
equilibrium distribution function is a heavy-tailed L\'evy-type
distribution with infinite variance.  We prove that solutions in an
appropriate domain of attraction of the equilibrium distribution
converge to solutions of a Fokker-Planck equation with a fractional
diffusion operator.
\medskip

 }
\end{center}

\section{Introduction}

The inelastic Kac model has been introduced in \cite{PT04}, with the
aim of obtaining a physically consistent one-dimensional dissipative
kinetic model, sufficiently rich to exhibit a variety of steady
states and similarity solutions. This model can be viewed as a
dissipative generalization of the Kac caricature of a Maxwell gas
introduced in the fifties \cite{Kac}. Kac model has been fruitfully
used from that time on, to find explicit rates of convergence
towards the Maxwellian equilibrium \cite{McK, TV}, since its simple
structure (with respect to the full Boltzmann equation) makes
possible to carry out exact computations. The inelastic Kac model
reads
\begin{equation} \label{boltz}
\partial_t f(v,t) = Q_p(f,f)(v,t),\quad v\in\R,\, t\geq 0
\end{equation}
where the right-hand side of \fer{boltz} describes the rate of
change of the density function $f$ due to dissipative collisions,
 \be\label{collint}
 Q_p(f,f)(v,t)= \int_{{\mathbb R}\times
[-\pi/2,\pi/2]} b (\theta) \left [\chi^{-1}
f(v_p^{**},t) f(w_p^{**},t) -f(v,t) f(w,t) \right ]\, {\d\theta}\,\d w.
 \ee
The kernel $b$ is an integrable function taking values in
$[-\pi/2, \pi/2]$, which describes the details of the possible
outcomes in the binary collision. The velocities $(v^{**}_p,
w^{**}_p)$ are the pre collision velocities of the so--called
inverse collision, which results with $(v,w)$ as post collision
velocities. Given $(v,w)$, the post collision velocities $(v^*_p,
w^*_p)$ are defined simply generalizing the Kac rule
 \be\label{colli2}
 v^*_p = v\cos\theta|\cos\theta|^p - w\sin\theta|\sin\theta|^p ,
 \quad
 w^*_p = v\sin\theta|\sin\theta|^p + w\cos\theta|\cos\theta|^p.
  \ee
In (\ref{colli2}) the positive constant $p < +\infty$ measures the
degree of inelasticity. If $p=0$, the binary collision is elastic
and we obtain the classical Kac equation, where the post collisional
velocities are given by a rotation in the $(v,w)$ plane. The factor
$\chi = |\sin\theta|^{2+2p} + |\cos\theta|^{2+2p}$  in \eqref{collint}
appears
from the Jacobian of the transformation
$\d v_p^{**}\d w_p^{**}$ into $\d v\d w$. The loss of energy in a single binary
collision depends on the choice of the {inelasticity parameter} $p$,
and it is given by
\begin{equation}\label{lost}
(v^*_p)^2 + (w^*_p)^2 = (v^2 + w^2)\left( |\sin\theta|^{2+2p} +
|\cos\theta|^{2+2p}\right).
\end{equation}
The structure of the inelastic Kac equation is similar to the
inelastic Boltzmann equation for a Maxwell gas, and also here mass
is conserved, while energy is non-increasing.

Resorting to Bobylev's argument \cite{Bob}, the dissipative Kac
equation can be fruitfully written in Fourier variables as
\begin{equation} \label{eq}
\partial_t \widehat f(\xi, t)  =
\widehat{Q}_p\left( \ff,\ff \right)(\xi,t) ,
\end{equation}
where $\ff(\xi,t)$ is the Fourier transform of $f(v,t)$ with respect to $v$
\[ \widehat{f}(\xi,t) = \int_{{\mathbb R}} e^{-i \xi v}\, f(v,t)\, \d v, \]
and
\begin{equation}\label{trascoll}
  \widehat{Q}_p\left( \ff,\ff \right)(\xi,t) =
\int_{-\frac \pi 2}^{\frac \pi 2}b (\theta)
 \left [ \widehat{f}(\xi_p^+,t)
\widehat{f}(\xi_p^-,t) - \widehat{f}(\xi,t) \widehat{f}(0,t) \right ]\,{\d\theta}.
\end{equation}
In \fer{trascoll}
\begin{equation}\label{csi}
\xi_p^+ = \xi \cos\theta|\cos\theta|^p , \quad \xi_p^- = \xi
\sin\theta|\sin\theta|^p.
\end{equation}
The Fourier description allows  to verify that the inelastic Kac
model, unlike what happens for the elastic model, possesses
Maxwellian equilibria with infinite energy. Indeed, the collision
operator \fer{trascoll} vanishes by choosing
 \be\label{gen-max}
 \widehat M_p(\xi) = \exp \left\{- \alpha |\xi|^{2/(1+p)} \right\} , \qquad \alpha >0.
 \ee
In case  $0<p\le 1$, the function \fer{gen-max} represents a  L\'evy
symmetric stable distribution of order $2/(1+p)$ \cite{Fe71}.
Consequently, the problem of convergence of the solution to
\fer{boltz}  to these equilibria for large times is deeply connected
with the central limit theorem for stable laws, like the classical
central limit theorem is closely connected with the convergence
towards equilibrium of the elastic Boltzmann or Kac equation
\cite{LR}. The main result in \cite{PT04} is that these Maxwellian
equilibria \fer{gen-max} only attract  solutions corresponding to
initial data which are in a suitable (small) domain of attraction of
\fer{gen-max}. In particular, the energy of these initial data has
to be unbounded.

As a matter of fact, the classical problem of the cooling of the
dissipative gas is concerned with initial densities of finite energy
\cite{Vil06}. In this relevant physical case, the stationary
solution is a Dirac delta function concentrated in $v=0$, which is
nothing but a particular case of equilibrium \fer{gen-max}
corresponding to $\alpha = 0$. This equilibrium is strongly
attractive, and its domain of attraction is given by all initial
densities with bounded energy \cite{Vil06}.

If on the contrary the energy of the initial datum is unbounded, the
collision mechanism given by \fer{colli2} is in general not enough
to cool down the gas towards a density with zero energy. In this
second case however, convergence to  Maxwellian equilibria of type
\fer{gen-max} is possible. Hence, the existence of these equilibria
is strictly linked to the \emph{weakness} of the particular
one-dimensional collision me\-cha\-nism of the dissipative Kac
equation.

The {\it grazing limit} procedure consists in letting the kernel $b
(\theta)$ to concentrate on $\theta =0$  (which implies $v_p^*=v$,
$w_p^* =w$) {in such a way that the contribution of the collision
integral does not vanish}. This can be done by assuming on a family
$\{b_\eps\}$ of collision kernels that satisfy the following
properties \cite{Vil2}

\begin{defi}[The grazing sequence]\label{def1}
 The sequence $\{b_\eps(\theta) \}$ is a
grazing collision sequence if
\begin{enumerate}[(a)]
   \item $b_{\eps}(\theta)= b_{\eps}(|\theta|) \geq 0$; \label{prima}
   \item $\supp\{ b_\eps\} \subset \{\theta \in [-\frac \pi 2,\frac \pi 2]:
   \ 0<c_\eps \leq |\theta| \leq d_\epsilon\}$ where $d_\eps \to 0$ for $\eps\to 0$; \label{seconda}
   \item $\int_0^{\frac \pi 2} b_\eps(\theta)\sin^{2} \theta \, \d \theta =1$ for all $\eps$. \label{terza}
  \end{enumerate}
\end{defi}

The physical meaning of condition \fer{terza} is related to the fact
that the cooling of the dissipative gas has to be guaranteed in the
limit. By means of \fer{lost}, it follows in fact that the energy
varies in time according to
 \be\label{en}
 \begin{aligned}
\frac \d{\d t} \int_{{\mathbb R}}  v^2 f(v,t)\,\d v &=
 \frac 12 \int_{{\mathbb R}^2}
\int_{-\pi/2}^{\pi/2} b_\eps (\theta)\, \left( (v_p^*)^2
+(w_p^*)^2 -v^2 - w^2 \right)\, f(v,t)f(w,t)\, {\d\theta}\,\d w\, \d v\\
&=- \left(\int_{-\pi/2}^{\pi/2} b_\eps (\theta) \left( 1-
|\sin\theta|^{2+2p} - |\cos\theta|^{2+2p}  \right)\, {\d\theta}\right )
\int_{{\mathbb R}}
 v^2 f(v,t)\,\d v.
\end{aligned}
 \ee
Hence, the grazing limit procedure has to be chosen so that the
(positive) coefficient
 \[
 L_\eps  = \int_{-\pi/2}^{\pi/2} b_\eps (\theta)
 \left(1 -|\sin\theta|^{2+2p} - |\cos\theta|^{2+2p} \right)\,{\d\theta}
 \]
remains bounded in the limit.
It is immediate to conclude that, for $0 < p \le 1$
 \[
1 - |\sin\theta|^{2+2p} - |\cos\theta|^{2+2p} \le 1 -
|\sin\theta|^{4} - |\cos\theta|^{4} = 2\sin^2\theta\cos^2\theta.
 \]
Moreover, a simple analysis of the behavior of the function
 \[
\Phi(y)= 1- y^{1+p} -(1-y)^{1+p}  -c_py(1-y), \qquad 0 \le y \le 1
 \]
shows that $\Phi(y)$ is non-negative at least whenever $c_p \le
p(1+p)2^{1-p}$. Consequently, we have the bounds
 \[
c_p \int_{-\pi/2}^{\pi/2}  b_\eps (\theta)\sin^2\theta
\cos^2\theta\,\d\theta \le L_\eps  \le 2 \int_{-\pi/2}^{\pi/2}
b_\eps (\theta)\sin^2\theta \cos^2 \theta \, \d\theta
 \]
which imply, under condition \fer{terza} of Definition \ref{def1},
that the decay of energy remains well-defined in the limit
procedure.

 In the case of the
elastic Kac equation, the grazing collision asymptotic has been
studied in \cite{to1}, resorting to the non cut-off formulation by
Desvillettes \cite{De}, in which the kernel $b $ is assumed to be
non integrable. For a large class of initial densities, essentially
all initial densities with finite energy, the solution to the
kinetic model has been proven in \cite{to1} to converge towards the
solution to the linear Fokker-Planck equation
 \be\label{FP}
\derpar{f}{t} = \frac{\partial^2 f }{\partial v^2} + \frac{\partial
}{\partial v}\left(vf\right).
 \ee
In this paper we will address a similar problem, namely the grazing
collision asymptotic of the inelastic Kac model, in the case in
which the kernel $b $ is an integrable function, the initial data
possess infinite energy and lie in a suitable neighborhood of a L\'evy
distribution of type \fer{gen-max}.

Our main result (Theorem \ref{teo1}) consists in proving that, in the inelastic regime
cha\-rac\-terized by $0< p \le 1$,  the solution to the inelastic Kac
model, corresponding to this choice of initial values, converges
towards the solution of a Fokker-Planck equation with a fractional
diffusion
  \be\label{FPF}
\derpar{f}{t} = 2\left( \alpha\, D^q f + \frac 1 q
 \frac{\partial }{\partial v}\left( vf\right)\right) , \qquad
q = \frac 2{1+p},
 \ee
where the fractional derivative of order $q$ is defined for instance
by the Fourier formula
 \be\label{fder}
D^q f  = \F^{-1}\left( -|\xi|^q \ff(\xi)\right),
 \ee
and $\F^{-1}$ stands for the inverse Fourier transform.

Fokker-Planck type equations with fractional diffusion appear in
many physical contexts \cite{Cha,Ben, Shu, MFL}, and have been
intensively studied both from the modeling and the qualitative point
of view. Likewise, L\'evy distributions have been applied to the
description of many physical processes, including turbulent flows
\cite{Shl}, diffusion in complex systems \cite{Ott}, chaotic
dynamics of classical conservative systems \cite{SZK93, KZB}, and
others.

Nevertheless, the connections between fractional diffusion equations
and kinetic models of Boltzmann type have been analyzed only
recently. In the framework of kinetic theory and asymptotic limits,
some insight on this connection has been recently done by Mellet,
Mischler and Mouhot \cite{Mel}. The starting point of their analysis
was the derivation of diffusion--type equations from a linear
Boltzmann equation describing the interactions of the particles with
the surrounding medium. Despite the classical problem, in which the
Maxwellian distribution of the background decays exponentially fast
at infinity, in the case described in \cite{Mel}, the Maxwellian
distribution is heavy-tailed, and the resulting diffusion limit
corresponds to a fractional diffusion.

\section{Preliminary results}

In this short section we resume the main results relative to the
dissipative Kac equation \fer{boltz}. The largest part of these
results have been obtained in \cite{PT04}, where the kernel
$b$, like in the original Kac model \cite{Kac}, was
assumed constant.  The analysis of \cite{PT04}, however, can be
easily extended to integrable kernels.

\begin{theo}[Existence, uniqueness and conservation laws]
Let $0<p\leq 1$. Let the initial datum $f_0 \geq 0$ satisfy the
(normalization) assumptions
\be\label{norm}
 \int_{\R} f_0(v) \, \d v =1,\quad \int_{\R}   v
\,f_0(v)\,\d v=0.
 \ee
 Then, the initial value problem for the dissipative Kac equation
\fer{boltz}--\fer{collint}, with $b\in L^1 ([- \pi/ 2,
\pi/2])$, has a unique non negative solution $f \in
C^1\left([0,\infty), L^1(\R)\right)$. This solution satisfies for
all $t>0$:
\[
\int_{\R} f(v, t) \, \d v =1,\quad \int_{\R} v\,f(v,t)\, \d v=0.
\]
\end{theo}
In this existence result, the initial data are not supposed to satisfy
$\int_\R |v| f_0(v)\, \d v <+\infty$ and so $\int_\R vf_0(v)\, \d v =0$ has to be
interpreted as a principal value integral.
The proof of this theorem follows along the same lines of the proof
for the conservative Kac equation, which goes back to Morgenstern
\cite{Mor1}, \cite{Mor2}.

Let us now consider the Fourier formulation
\fer{eq}--\fer{trascoll}, and let us suppose that the kernel $b$ is
an even function. Then, Equations \fer{eq}--\fer{trascoll} can be
written as
 \be\label{four-sym}
\partial_t \widehat f (\xi, t)=  \int_0^{\frac \pi 2} b(\theta) \left( \widehat f(\xi \cos^{p+1} \theta,t)
 \left(\widehat f (\xi \sin^{p+1} \theta,t)+ \widehat f (-\xi \sin^{p+1} \theta,t)\right ) -2\widehat f (\xi,t)\right )\, \d \theta.
 \ee
 Due to the integrability of the collision kernel, we can split
the collision integral to obtain
\[
\partial_t \widehat f (\xi, t)=  \int_0^{\frac \pi 2} b(\theta)
\widehat f(\xi \cos^{p+1} \theta,t) \left(\widehat f (\xi \sin^{p+1} \theta,t)+
\widehat f (-\xi \sin^{p+1} \theta,t\right )\,  \d \theta
-{\sigma } \widehat f (\xi,t)
\]
where
\[
{\sigma } = \int_{-\frac \pi 2}^{\frac \pi 2} b(\theta)\, \d
\theta.
\]
The (unique) solution $\widehat f(t)$ of Equation \fer{four-sym} can be
explicitly written using the so-called Wild expansion, found by Wild in \cite{wil} and exploited
extensively since then by Bobylev (\cite{Bob}) and many others.
It reads
\be\label{wild}
\widehat f (\xi, t)= \e^{-{\sigma } t} \sum_{n=0}^\infty
\varphi_n(\xi) (1-e^{-{\sigma } t})^n,
\ee
where
\[
\begin{aligned}
&\varphi_0(\xi) = \widehat f_0(\xi),\\
&\varphi_{n+1}(\xi) = \frac 1{n+1} \sum_{j=0}^n  Q_+(\varphi_j,
\varphi_{n-j})(\xi)
\end{aligned}
\]
and
\begin{multline*}
Q_+(\phi,\psi)(\xi) = \frac 1{2{\sigma }} \int_0^{\frac \pi 2}
b(\theta) \left[
\phi(\xi \cos^{p+1} \theta)\left( \psi(\xi \sin^{p+1} \theta) +
\psi(-\xi \sin^{p+1} \theta)\right )  \right .\\
\left.+ \psi(\xi \cos^{p+1} \theta)\left( \phi(\xi \sin^{p+1}
\theta) + \phi(-\xi \sin^{p+1} \theta)\right ) \right ]\, \d \theta.
\end{multline*}
Wild expansion follows from an iteration formula based on the
solution written by Duhamel's formula
\[
\widehat f(\xi, t)= e^{-{\sigma } t} \widehat f_0(\xi) + {\sigma }
\int_0^t e^{-{\sigma } (t-s)} Q_+(\widehat f, \widehat f)(\xi, s)\, \d
s.
\]
For any initial
density $\widehat f_0$ the solution $\widehat f(t)$ is the limit of the sequence
\be\label{succ-duhamel}
  \widehat f^{(n)}(\xi, t)=
   e^{-{\sigma} t} \widehat f_0(\xi) + \sigma
\int_0^t e^{-{\sigma} (t-s)} Q_+(\widehat f^{(n-1)}, \widehat f^{(n-1)})(\xi, s)\, \d \theta\, \d s
\ee
with $\widehat f^{(0)}(\xi,t)=\widehat f_0(\xi)$.
It should be noticed however that the term $\widehat f^{(n)}$ does not correspond in general to a partial sum of \eqref{wild}.
We will exploit in this paper the representation through Wild sums since it  better suits our goals.

In what follows,  we will assume that the initial data $f_0$ are
even functions.
Therefore,  the solution itself is even together with its
Fourier transform and Equation \eqref{four-sym} reads
\be
 \label{four-even}
\partial_t \widehat f(\xi, t)= 2\int_{0}^{\frac \pi 2} b(\theta)
\left( \widehat f(\xi \cos^{p+1} \theta ,t) \widehat f(\xi \sin^{p+1} \theta,t) -\widehat f(\xi,t)\right )\, \d \theta.
 \ee
This assumption allows to simplify many details of the forthcoming
proofs. It has to be noted, however, that the results continue to
hold with minor modifications for general initial data which have
first momentum equal to zero, at the price of an increasing number
of computations.

A further result in \cite{PT04} is concerned with the large time
behavior of solutions.

\begin{theo}[\cite{PT04}]\label{PT}
Let $0<p\leq 1$ and let $f(t)$ be the unique solution of the
dissipative Kac equation \fer{boltz}--\fer{collint}, corresponding
to the initial density $f_0$ satisfying the normalization conditions
\fer{norm}, and such that, for some $0<\delta\leq {2p}/(p+1)$
\be\label{mom-delta} \int_\R |v|^{\frac 2{p+1}+\delta}
|f_0(v)-M_p(v)|\, \d v<+\infty. \ee Then,
\[
\lim_{t\to +\infty} \sup_{\xi \neq 0} \frac {|\widehat f(\xi, t)-\widehat
M_p(\xi)|} {|\xi |^{\frac 2{p+1}+\delta} }=0.
\]
\end{theo}

\begin{rem}
Condition \eqref{mom-delta}, which is a condition on moments of the
initial data, can be replaced in the proof of Theorem \ref{PT} by
the weaker condition
\[
\sup_{\xi \neq 0} \frac {|\widehat f_0(\xi)-\widehat M_p(\xi)|}
{|\xi |^{\frac 2{p+1}+\delta} } <+\infty .
\]
\end{rem}

\section{Convergence to the solution of the Fokker--Planck equation}\label{conver}

In this section we state and prove our main result on the grazing
collision limit of solutions of the dissipative Kac equation
\eqref{boltz}--\eqref{collint} towards the solution of a
Fokker-Planck equation with fractional diffusion \eqref{FPF}.

In order to show why the Fokker-Planck equation \eqref{FPF} is the
result of the grazing procedure, the following  computation on the
steady state $\widehat M_p(\xi) = \exp \left\{- \alpha
|\xi|^{2/(1+p)} \right\}$ with $\alpha >0$ will be useful.

 Since
$\widehat M_p(\xi)$ satisfies the Kac equation \fer{eq}
\[
\partial_t \widehat M_p(\xi)= 2\int_{0}^{\frac \pi 2} b_\eps(\theta)
\left( \widehat M_p(\xi \cos^{p+1} \theta) \widehat M_p(\xi
\sin^{p+1} \theta) -\widehat M_p(\xi)\right )\, \d \theta.
\]
Let us  expand $\widehat M_p$ around the origin and around the point
$\xi$ respectively. Then
\[
\begin{aligned}
&\widehat M_p(\xi \sin^{p+1} \theta)  = 1- \alpha |\xi|^{2/(p+1)} \sin^2 \theta + o (|\xi|^{2/(p+1)} \sin^2 \theta), \\
& \widehat M_p(\xi \cos^{p+1} \theta)  = \widehat M_p(\xi) +
\partial_\xi \widehat M_p(\xi) \xi(\cos^{p+1}\theta-1)  + o (\xi
(\cos^{p+1}\theta-1)).
\end{aligned}
\]
Using these expressions we obtain
\[
\partial_t \widehat M_p(\xi) =
2\int_{0}^{\frac \pi 2} b_\eps(\theta) \left(- \alpha |\xi|^{2/(p+1)}
\widehat M_p(\xi) \sin^2 \theta + \partial_\xi \widehat M_p(\xi) \xi(\cos^{p+1}\theta-1) + R(\theta, \xi)
\right )\, \d \theta.
\]
The reminder term $R(\theta, \xi)$ behaves like $\sin^{2+\mu}
\theta$ for $\mu >0$ when $\theta \to 0$. Now, the grazing
conditions of Definition \ref{def1} imply $\int b_\eps (\theta)
\sin^{2+\mu} \theta \, \d \theta \to 0$ as $\eps \to 0$ and passing
to the limit in $\eps$ we get (at least in a formal way)
\[
\partial_t \widehat M_p(\xi) = - 2\alpha |\xi|^{2/(p+1)} \widehat M_p(\xi) - (p+1) \xi \partial_\xi \widehat
M_p(\xi),
\]
which is exactly the Fokker-Planck equation \eqref{FPF} in the Fourier variable.
Our goal is to make this computation rigourous for any solution with initial data suitably close to the steady state $M_p$.

 It is easy to see, as it was done in \cite{CT1} for the classical Fokker--Planck equation, that the solution $f(t)$ of Equation
 \eqref{FPF} has an explicit expression in terms of a convolution
 between the initial data and the stationary state. In the Fourier
 variable this solution reads
\be\label{sol-FP-four}
  \widehat f(\xi, t)= \widehat f_0\left( \xi e^{-(p+1)t}\right ) e^{-\alpha |\xi|^{\frac
  2{p+1}}(1-e^{-2t})}.
 \ee
In the physical space \be\label{solFP}
 f(v,t)= \frac 1{\beta(t)} f_0\left( \frac \cdot {\beta (t)}\right ) \ast \frac 1{\gamma(t)} M_p \left( \frac \cdot {\gamma(t)}\right
 )(v),
\ee with
\[
 \beta (t)= e^{-(p+1)t}, \qquad \gamma(t)= (1-e^{-2t})^{\frac
{p+1}2}.
\]
Note that
 \[
 \beta(t)^{2/(p+1)} +  \gamma(t)^{2/(p+1)} = 1
 \]

\begin{theo} \label{teo1}
Assume $\{b_\eps(\theta)\}_{\eps >0} \subseteq  L^1([-\frac \pi
2,\frac \pi 2])$ be a family of collision kernels satisfying
Definition \ref{def1}. Let
 $0<p\leq 1$ and let
 $f_\eps(t) \in C^1([0,+\infty), L^1(\R))$ be the solutions of the dissipative Kac equations
\[
\partial_t \widehat f_\eps(\xi, t)= 2\int_{0}^{\frac \pi 2} b_\eps(\theta) \left( \widehat f_\eps(\xi \cos^{p+1} \theta ,t) \widehat f_\eps(\xi \sin^{p+1} \theta,t) -\widehat f_\eps(\xi,t)\right )\, \d \theta  \label{Beps}\\
\]
corresponding to an even initial density $f_0\geq 0$ satisfying the
normalization condition $\int f_0(v)\, \d v=1$.

Let us suppose in addition that the initial datum $f_0$ satisfies
the conditions
\begin{enumerate}[A)]
\item there is  $\alpha >0$ such that
 \[
\lim_{\xi \to 0^+} \frac{1-\widehat f_0(\xi)}{\xi^{\frac
2{p+1}}}=\alpha;
  \]
\item $\widehat f_0$ is differentiable outside the origin and  the function ${\partial_\xi\widehat f_0(\xi)}/{\xi^{\frac{1-p}{p+1}}}$
is uniformly $\delta$--H{\"o}lder continuous on bounded  subsets of $(0,+\infty)$, namely
there is $\delta \in (0,1)$ and for all $R>0$ there is  $K(R)> 0$
such that
\[
\sup_{0<\xi\leq R,\ 0<\tau\leq R,\ \xi \neq \tau} \frac{ \left|
\frac{\partial_\xi\widehat f_0(\xi)}{\xi^\frac{1-p}{p+1}} -
\frac{\partial_\xi\widehat f_0(\tau)}{\tau^\frac{1-p}{p+1}} \right|}
{|\xi -\tau|^{\delta}} \leq K(R).
\]
\end{enumerate}
Then, if $f (t)$  is  the solution \eqref{sol-FP-four} of  the
Fokker--Planck equation
\[
\partial_t \widehat f (\xi, t)= -2\alpha |\xi|^{\frac 2{p+1}}  \widehat f(\xi, t)- (p+1) \xi \partial_\xi \widehat f(\xi,
t),
\]
corresponding to the same initial datum $f_0$,
\[
 \lim_{\eps \to 0}\sup_{t\geq 0,\ \xi \neq 0} \frac{\left|\widehat f_\eps(\xi, t)-\widehat f(\xi, t)\right |}{|\xi|^{\frac 2{p+1}}}
 =0 \, .
\]

\end{theo}


We begin by proving that assumptions A) and B) on the initial data,
at least for $\eps$ small enough, are uniformly propagated along the
solutions $f_\eps(t)$ of the dissipative Kac equation. In fact we
have

\begin{lem}\label{Propag}
Assume $\{b_\eps(\theta)\}_{\eps >0} \subseteq  L^1([-\frac \pi
2,\frac \pi 2])$ be a family of collision kernels satisfying
pro\-per\-ties \fer{prima} and \fer{seconda} of Definition
\ref{def1}. Assume  $f_0\geq 0$ is an even function,  satisfying the
normalization condition $\int f_0(v)\, \d v=1$.
\begin{enumerate}[A)]
 \item
If there is  $\alpha >0$ such that
 \be\label{ip-lim}
\lim_{\xi \to 0^+} \frac{1-\widehat f_0(\xi)}{\xi^{\frac 2{p+1}}}=\alpha
  \ee
then the solutions $\widehat f_\eps(t)$ of the Kac equations
\fer{eq}--\fer{trascoll} with $f_0$ as initial datum satisfies the
same property uniformly in time and in $\eps$, i.e.
\begin{eqnarray}
&\dis{\lim_{\xi\rightarrow 0^+}\frac{1-\widehat f_\eps(\xi,t)
}{\xi^{\frac 2{p+1}}}=\alpha}\quad \text {uniformly in $\eps$ and
$t$.}\label{lim}
\end{eqnarray}
\item If in addition $\widehat f_0$ is differentiable outside the origin and  the
function ${\partial_\xi\widehat f_0(\xi)}/{\xi^{\frac{1-p}{p+1}}}$
is uniformly $\delta$--H{\"o}lder continuous on bounded  subsets of
$(0,+\infty)$, so that there is $\delta \in (0,1)$ and for all $R>0$
there is  $K(R)> 0$ such that
 \be\label{ip-holder}
 \sup_{\xi,\, \tau
\in (0,R],\ \xi \neq \tau} \frac{ \left| \frac{\partial_\xi\widehat
f_0(\xi)}{\xi^\frac{1-p}{p+1}} - \frac{\partial_\xi\widehat
f_0(\tau)}{\tau^\frac{1-p}{p+1}} \right|} {|\xi -\tau|^{\delta}}
\leq K(R),
 \ee
then  the same is true uniformly for  $\widehat f_\eps (t)$, namely there is $\tilde K(R)\geq K(R)$ such
that
\be\label{holder}
\sup_{t\geq 0}\quad \sup_{\xi,\, \tau \in (0,R],\
\xi \neq \tau} \frac{ \left| \frac{\partial_\xi \widehat f_\eps(\xi,
t)}{\xi^\frac{1-p}{p+1}} - \frac{\partial_\xi\widehat f_\eps(\tau,
t)}{\tau^\frac{1-p}{p+1}} \right|} {|\xi -\tau|^{\delta}} \leq
\tilde K(R),\quad \text{for $\eps$ small enough.}
\ee
\end{enumerate}
\end{lem}

\noindent {\bf Proof  of Lemma \ref{Propag}:}
Since the stationary solution $\widehat M_p(\xi) = e^{-\alpha
|\xi|^{\frac 2{p+1}}}$ satisfies \eqref{ip-lim}, for all $\eta>0$
there exists $\lambda
>0$ such that
\be\label{A)}
 \sup_{0<\xi<\lambda}\frac {\left|\widehat M_p(\xi)-\widehat
f_0(\xi)\right|}{\xi^{\frac 2{p+1}}}<\eta .
 \ee
Hence it is enough to prove that for all $\eps>0$ and $t>0$
\be\label{M-lim}
\sup_{0<\xi<\lambda}\frac {|\widehat M_p(\xi)-\widehat
f_\eps(\xi,t)|}{\xi^{\frac 2 {p+1}}}<\eta.
 \ee
The unique solution $\widehat f_\eps(\xi,t)$ can be expressed using
the Wild expansion (cfr. equation \fer{wild}). In the particular
case of even initial data we have
 \[
\widehat f_\eps(\xi, t)= \e^{-{\sigma_\eps } t}
\sum_{n=0}^\infty \varphi_n^\eps(\xi)
(1-e^{-{\sigma_\eps } t})^n,
\]
where
\[
 \sigma_\eps=
2\int_0^{\frac \pi 2} b_\eps(\theta)\, \d \theta,
\]
\[
\begin{aligned}
&\varphi_0^\eps(\xi) = \widehat f_0(\xi),\\
&\varphi_{n+1}^\eps(\xi) = \frac 1{n+1} \sum_{j=0}^n
Q_+^\eps(\varphi_j^\eps, \varphi_{n-j}^\eps)(\xi),
\end{aligned}
\]
and
\[
Q_+^\eps(\phi,\psi)(\xi) = \frac 1{{\sigma_\eps }} \int_0^{\frac \pi
2} b_\eps(\theta) \left[ \phi(\xi \cos^{p+1} \theta) \psi(\xi
\sin^{p+1} \theta) + \psi(\xi \cos^{p+1} \theta) \phi(\xi \sin^{p+1}
\theta)\right ] \, \d \theta.
\]
Since for all $t>0$
\[
\e^{-{\sigma_\eps } t} \sum_{n=0}^\infty
 (1-e^{-{\sigma_\eps } t})^n=1,
 \]
inequality \eqref{M-lim} follows provided
 \be\label{ric}
\sup_{0<\xi<\lambda}\frac {|\widehat M_p(\xi)-
\varphi_n^\eps(\xi)|}{\xi^{\frac 2 {p+1}}}<\eta
 \ee
uniformly in $n$.  It is enough to prove that \fer{ric} holds for
$\varphi_1^\eps$.  Then  by a recursive argument \eqref{ric} holds
for any $n>1$. Since
 \[
 \widehat M_p(\xi)= \frac2{\sigma_\eps} \int_0^{\frac\pi
2}b_\eps(\theta)\widehat M_p(\xi\cos^{p+1} \theta) \widehat M_p(\xi
\sin^{p+1} \theta)\, \d \theta \, ,
\]
 we have
\[
 \begin{aligned}
\varphi_{1}^\eps(\xi)-\widehat M_p(\xi) &= \frac 2{\sigma_\eps}
 \int_0^{\frac\pi 2}b_\eps(\theta)\left[\widehat f_0(\xi \cos^{p+1} \theta)
\widehat f_0(\xi \sin^{p+1} \theta)-\widehat M_p(\xi \cos^{p+1} \theta) \widehat M_p(\xi \sin^{p+1} \theta)\right]\,
\d \theta\\
&=  \frac 2{\sigma_\eps}\int_0^{\frac\pi 2}b_\eps(\theta)\,
\left\{\left[\widehat f_0(\xi \cos^{p+1}\theta)-\widehat M_p(\xi
\cos^{p+1} \theta)\right]
\widehat f_0(\xi \sin^{p+1} \theta) \right .\\
&\quad\quad\quad\quad+\left. \left[\widehat f_0(\xi \sin^{p+1}\theta)-\widehat
M_p(\xi \sin^{p+1} \theta)\right  ]\widehat M_p(\xi \cos^{p+1}
\theta)\right \}\, \d \theta.
\end{aligned}
\]
Using the property $|\widehat f_0(\xi)| \le 1$, we obtain
\[
\begin{aligned}
\frac{\left|\varphi_1^\eps(\xi)-\widehat M_p(\xi)\right|}{\xi^{\frac
2{p+1}}}&\leq \frac 2{\sigma_\eps}\int_0^{\frac\pi
2}b_\eps(\theta) \left( \cos ^2 \theta\frac{\left|\widehat f_0(\xi
\cos^{p+1} \theta)-
\widehat M_p(\xi \cos^{p+1} \theta)\right|}{\left(\xi \cos^{p+1} \theta\right)^{\frac 2{p+1}} }\right .\\
&\qquad \qquad \left. +\sin^{2} \theta\ \frac{ \left|
\widehat f_0(\xi \sin^{p+1} \theta)-\widehat M_p(\xi \sin^{p+1} \theta
)\right|}
{\left(\xi \sin^{p+1} \theta\right )^{\frac 2{p+1}}}\right) \, \d \theta.\\
\end{aligned}
\]
Finally, since $0<\xi \cos^{p+1}\theta<\xi$ and
$0<\xi\sin^{p+1}\theta<\xi$, by \eqref{A)}
$$
\sup_{0<\xi<\lambda}\frac{\left|\varphi_{1}^\eps(\xi,t)-\widehat
M_p(\xi)\right|}{\xi^{\frac 2{p+1}}}\leq  \frac
{2\lambda}{\sigma_\eps} \int_0^{\frac \pi 2}b_\eps(\theta) \d\theta =\lambda.
$$
This concludes the proof of part $A)$. To prove  $B)$ consider that
conditions \eqref{ip-lim} and \eqref{ip-holder} on the initial data
imply \be\label{lim-der}
 \lim_{\xi \to 0^+}\frac {\partial_\xi \widehat f_0(\xi)}{ \xi^{\frac {1-p}{p+1}}} =-\frac{2\alpha}{p+1}.
\ee Indeed, by \eqref{ip-lim} and Cauchy theorem,  there is
$\{\tilde \xi_n\}$ such that $\tilde \xi_n \to 0^+$ for $n\to
+\infty$ and
\[
 \lim_{n\to +\infty} \frac {\partial_\xi \widehat f_0(\tilde \xi_n)}{\tilde \xi_n^{\frac {1-p}{p+1}}} =-\frac{2\alpha}{p+1}.
\]
Together with the H\"older continuity outside the origin
\eqref{ip-holder}, this leads to \eqref{lim-der}. Therefore
\[
 \left|\frac {\partial_\xi \widehat f_0(\xi)}{ \xi^{\frac {1-p}{p+1}}} +\frac{2\alpha}{p+1}\right | \leq K(R) |\xi|^\delta, \quad \xi \in
 (0,R],
\]
and this implies
\be\label{ip-mag-der}
 \sup_{0<\xi\leq R}  \frac{\left| \partial_\xi \widehat f_0(\xi)\right |}{\xi^\frac{1-p}{p+1}} \leq \frac{2\alpha}{p+1} + K(R) R^\delta :=K_1(R) >0.
\ee
First of all, we prove that condition \eqref{ip-mag-der} is uniformly propagated on any $\varphi_n^\eps$. For
this purpose, it is enough to prove it for $\varphi_1^\eps$.
Recall that
\[
\phi_1^\eps(\xi)= \frac 2 {\sigma_\eps} \int_0^{\frac \pi 2} b_\eps
(\theta) \widehat f_0(\xi \cos^{p+1} \theta) \widehat f_0(\xi
\sin^{p+1} \theta)\,  \d \theta,
\]
and denote
\[
 F_0(\xi) = \frac {\partial_\xi \widehat f_0(\xi)}{\xi^{\frac {1-p}{p+1}}}.
\]
By \eqref{ip-mag-der} and Lebesgue theorem, $\partial_\xi \varphi_1^\eps (\xi)$ exists for all $\xi \neq 0$
and
\begin{multline*}
\partial_\xi \varphi_1^\eps (\xi) = \frac 2{\sigma_\eps} \int_0^{\frac \pi 2} b_\eps(\theta)\left[
\cos^{p+1}\theta \ \partial_\xi
\widehat f_0 (\xi \cos^{p+1}\theta) \widehat f_0 (\xi \sin^{p+1}\theta) \right .\\
\left.
+\sin^{p+1}\theta\  \partial_\xi
\widehat f_0 (\xi \sin^{p+1}\theta) \widehat f_0 (\xi \cos^{p+1}\theta) \right ]\, \d \theta.
\end{multline*}
So, for $0<\xi\leq R$ we have
\begin{multline*}
\frac{\partial_\xi \varphi_1^\eps (\xi)}{\xi^{\frac{1-p}{p+1}}} = \frac 2{\sigma_\eps} \int_0^{\frac \pi 2} b_\eps(\theta)\left[
\cos^{2}\theta \ F_0 (\xi \cos^{p+1}\theta) \widehat f_0 (\xi \sin^{p+1}\theta) \right .\\
\left.
+\sin^{2}\theta\  F_0 (\xi \sin^{p+1}\theta) \widehat f_0 (\xi \cos^{p+1}\theta) \right ]\, \d \theta.
\end{multline*}
Now, using  \eqref{ip-mag-der} and the property $|\widehat f_0(\xi)|\leq 1$ we get
\[
 \frac{\left|\partial_\xi \varphi_1^\eps (\xi)\right |}{\xi^{\frac{1-p}{p+1}}} \leq
   \frac {2K_1(R)}{\sigma_\eps}\int_0^{\frac \pi 2} b_\eps(\theta)\, \d \theta =K_1(R).
\]
By a recursive procedure, we get therefore for all $n>1$ and $0<\xi\leq R$
\be\label{ip-mag-der-n}
 \frac{\left|\partial_\xi \varphi_n^\eps (\xi)\right |}{\xi^{\frac{1-p}{p+1}}} \leq K_1(R)
\ee
and through  Wild expansion \eqref{wild}, we get for all $t\geq 0$ and $0<\xi\leq R$
\be\label{ip-mag-der-sol}
\frac{\left|\partial_\xi \widehat f_\eps (\xi,t)\right |}{\xi^{\frac{1-p}{p+1}}} \leq K_1(R).
\ee
Let us come now to the proof of \eqref{holder} and, as we did before, we recover the result for the
first term $\varphi_1^\eps$.
We have to prove that there is $\tilde K(R) \geq K(R)$ such that for $\xi, \tau \in (0,R]$ and at least for $\eps$ small enough
\[
\left| \frac{\partial_\xi\phi_1^\eps(\xi)}{\xi^{\frac {1-p}{p+1}}} -
\frac{\partial_\xi\phi_1^\eps(\tau)}{\tau^{\frac {1-p}{p+1}}} \right
| \leq \tilde K(R) |\xi -\tau|^{\delta},
\]
where
\begin{multline*}
 \frac{\partial_\xi\phi_1^\eps(\xi)}{\xi^{\frac {1-p}{p+1}}} - \frac{\partial_\xi\phi_1^\eps(\tau)}{\tau^{\frac {1-p}{p+1}}}
 = \\
 \frac 2 {\sigma_\eps} \int_0^{\frac \pi 2} b_\eps (\theta)
 \left[
 \cos^2\theta\,
 F_0(\xi \cos^{p+1} \theta)
\widehat f_0(\xi \sin^{p+1} \theta) +\sin^2\theta \, F_0(\xi \sin^{p+1}
\theta)
\widehat f_0(\xi \cos^{p+1} \theta)\right .\\
\left . - \cos^2\theta\,
 F_0(\tau\cos^{p+1} \theta)
\widehat f_0(\tau \sin^{p+1} \theta) -\sin^2\theta\,
 F_0(\tau \sin^{p+1} \theta)
\widehat f_0(\tau \cos^{p+1} \theta) \right ] \,  \d \theta\, .
 \end{multline*}
Since $f_0$ satisfies \eqref{ip-mag-der}, by Cauchy theorem we
get for all $\xi$ and $\tau$ in $[0,R]$, $\xi \neq \tau$
\[
\left|\frac{\widehat f_0(\xi) -\widehat f_0(\tau)}{\xi^{\frac
2{p+1}}-\tau^{\frac 2 {p+1}}} \right |= \frac {p+1}2 \left|
\frac{\partial_\xi \widehat f_0(\bar \xi)}{\bar \xi^{\frac
{1-p}{p+1}}}\right | \leq \frac {p+1}2 K_1(R) :=K_2(R)
\]
with $\bar \xi$ between $\xi$ and $\tau$. Consequently $\widehat
f_0$ satisfies
\be\label{hold-func}
\left|\widehat f_0(\xi)
-\widehat f_0(\tau)\right | \leq K_2(R) \left| \xi^{\frac
2{p+1}}-\tau^{\frac 2 {p+1}} \right|\quad \text{for all } \xi, \tau
\in [0,R].
\ee
Using \eqref{ip-holder},
\eqref{ip-mag-der} and \eqref{hold-func}, we get
\[
\begin{aligned}
&\left | \frac{\partial_\xi\phi_1^\eps(\xi)}{\xi^{\frac {1-p}{p+1}}} - \frac{\partial_\xi\phi_1^\eps(\tau)}{\tau^{\frac {1-p}{p+1}}} \right |\\
&\leq \frac 2 {\sigma_\eps} \int_0^{\frac \pi 2} b_\eps (\theta)
 \left \{ \cos ^2\theta \left[ \left| F_0(\xi \cos^{p+1} \theta) -F_0(\tau \cos^{p+1} \theta)\right | \left| \widehat f_0(\xi \sin^{p+1} \theta)\right | \right .\right .\\
& \hskip 3cm +\left. \left| F_0(\tau \cos^{p+1} \theta)\right |
\left|\widehat f_0(\xi \sin^{p+1} \theta)  - \widehat f_0(\tau \sin^{p+1} \theta)\right | \right ] \\
&\hskip 2.5cm +\sin^2 \theta  \left[ \left| F_0(\xi \sin^{p+1}
\theta) - F_0(\tau \sin^{p+1} \theta)  \right |
\left| \widehat f_0(\xi \cos^{p+1} \theta)\right |\right . \\
&\hskip 3cm + \left.\left. \left| F_0(\tau \sin^{p+1} \theta)\right
| \left| \widehat f_0(\xi \cos^{p+1} \theta)-   \widehat f_0(\tau \cos^{p+1}
\theta) \right | \right ] \right \}\,  \d \theta
\\
& \leq \frac 2 {\sigma_\eps} \int_0^{\frac \pi 2} b_\eps (\theta)
 \left \{
 \cos^2 \theta \left [K(R) \left| \xi \cos^{p+1} \theta -\tau\cos^{p+1} \theta\right |^{\delta}
 \right . \right .
\\
 & \hskip 3cm
 + \left. K_2(R)K_1(R) \left| (\xi \sin^{p+1} \theta)^{\frac 2{p+1}} -(\tau\sin^{p+1} \theta)^{\frac 2{p+1}} \right |\right ]
 \\
 &      \hskip 2.5cm
 +\sin^2 \theta\left [K(R) \left| \xi \sin^{p+1} \theta -\tau \sin^{p+1} \theta \right |^{\delta}\right .\\
 &  \left. \left. \hskip 3cm +
  K_2(R)K_1(R) \left| (\xi \cos^{p+1} \theta)^{\frac 2{p+1}} -(\tau\cos^{p+1}\theta)^{\frac 2{p+1}}\right |
  \right ]\right \} \, \d \theta\\
 & \leq \frac 2 {\sigma_\eps} \int_0^{\frac \pi 2} b_\eps (\theta)
 \left \{
 K(R) (\cos\theta)^{2+\delta(p+1)} \left| \xi -\tau\right |^{\delta}
 +  K_2(R)K_1(R) \cos^2 \theta \sin^2 \theta \left|\xi^{\frac 2{p+1}} -\tau^{\frac 2{p+1}} \right | \right . \\
 &   \hskip 3cm  \left .+ K(R) (\sin\theta)^{2+\delta(p+1)} \left| \xi -\tau\right |^{\delta} +
  K_2(R)K_1(R) \sin^2 \theta \cos^2 \theta  \left| \xi^{\frac 2{p+1}} -\tau^{\frac 2{p+1}}\right |
  \right \} \, \d \theta.
\end{aligned}
\]
Since
\[
\left| \xi^{\frac 2{p+1}} -\tau^{\frac 2{p+1}}\right | = \frac
2{p+1} |\bar \xi|^{\frac {1-p}{p+1}} |\xi-\tau|,
\]
with $\bar \xi$ between $\xi$ and $\tau$,  for $\xi$, $\tau \in (0,R]$, $\delta
\in (0,1)$ and $ C(R)>0$ suitably chosen, we get
\[
\left| \xi^{\frac 2{p+1}} -\tau^{\frac 2{p+1}}\right | \leq
C(R) |\xi-\tau|^{\delta}.
\]
It follows that
\begin{multline*}
 \left | \frac{\partial_\xi\phi_1(\xi)}{\xi^{\frac {1-p}{p+1}}} - \frac{\partial_\xi\phi_1(\tau)}{\tau^{\frac {1-p}{p+1}}} \right |
 \leq \left| \xi -\tau\right |^{\delta}  \frac 2 {\sigma_\eps} \int_0^{\frac \pi 2} b_\eps (\theta)
 \left \{
 K(R) \left(\cos^{2+\delta(p+1)}\theta + \sin^{2+\delta(p+1)}\theta\right ) \right . \\
 \left . + 2 C(R) K_2(R)K_1(R) \cos^2 \theta \sin^2 \theta \right \} \, \d
 \theta.
 \end{multline*}
By assuming $\tilde K(R) \geq K(R)$ large enough, it is not restrictive to assume
\[
 2 C(R) K_2(R)K_1(R) \leq \tilde K(R).
\]
If this is the case,
\begin{multline*}
 K(R) \left(\cos^{2+\delta(p+1)}\theta  + \sin^{2+\delta(p+1)}\theta \right )+ 2 C(R) K_2(R)K_1(R) \cos^2 \theta \sin^2 \theta\\
 \leq \tilde K(R) \left(\cos^{2+\delta(p+1)}\theta  + \sin^{2+\delta(p+1)}\theta +   \cos^2 \theta \sin^2 \theta\right ).
 \end{multline*}
On the other hand, for $\theta$ sufficiently close to zero,
\[
\cos^{2+\delta(p+1)}\theta  + \sin^{2+\delta(p+1)}\theta+   \cos^2
\theta \sin^2 \theta = 1-\frac {\delta(p+1)} 2 \theta^2  + o(\theta^2).
\]
Hence, there is $\bar \theta>0$
such that for $\theta \in [0,\bar \theta]$
\[
 \cos^{2+\delta(p+1)}\theta  + \sin^{2+\delta(p+1)} \theta+   \cos^2 \theta \sin^2 \theta  <1.
\]
By condition \eqref{seconda} of Definition \ref{def1}  for $\eps$
small enough
\[
\begin{aligned}
& \left | \frac{\partial_\xi\phi_1^\eps(\xi)}{\xi^{\frac {1-p}{p+1}}} - \frac{\partial_\xi\phi_1^\eps(\tau)}{\tau^{\frac {1-p}{p+1}}} \right |\\
&\leq \tilde K(R)  \left| \xi -\tau\right |^{\delta}  \frac 2
{\sigma_\eps} \int_0^{\bar \theta } b_\eps (\theta)
\left(\cos^{2+\delta(p+1)}\theta  + \sin^{2+\delta(p+1)}\theta +  \cos^2 \theta \sin^2 \theta \right ) \, \d \theta\\
&\leq  \tilde K(R)  \left| \xi -\tau\right |^{\delta} \frac 2
{\sigma_\eps} \int_0^{\bar \theta } b_\eps (\theta) \, \d \theta
\leq \tilde  K(R)  \left| \xi -\tau\right |^{\delta}.
\end{aligned}
\]
Finally
 \be\label{hol-phi1}
 \left | \frac{\partial_\xi\phi_1^\eps(\xi)}{\xi^{\frac {1-p}{p+1}}} - \frac{\partial_\xi\phi_1^\eps(\tau)}{\tau^{\frac {1-p}{p+1}}} \right |\\
\leq  \tilde K(R)  |\xi-\tau|^{\delta}.
\ee
Recursively, we can prove the same estimate for $\varphi_n^\eps$,
$n>1$ and therefore for the solution $\widehat f_\eps$. This concludes the proof of part $B)$. \hfill $\square$

\bigskip

We are now in a position to prove Theorem \ref{teo1}.

\vskip 5mm


\noindent {\bf Proof  of Theorem \ref{teo1}:} Let us first underline
again that if $f_0$ is even, then the solutions $f_\eps(t)$ and
$f(t)$ are even functions together with their Fourier transforms.
Our goal will be to prove that
\[
\lim_{\eps \to 0}\sup_{t\geq 0,\ \xi \neq 0} \frac{\left|\widehat
f_\eps(\xi, t)-\widehat f(\xi, t)\right |}{|\xi|^{\frac 2{p+1}}} =0.
\]
In more detail, we will show that for all ${\lambda} >0$ there is
$\bar \eps =\bar \eps ({\lambda})$ such that for all $0<\eps<\bar
\eps$
\be \sup_{t\geq 0,\ \xi \neq 0} \frac{\left|\widehat
f_\eps(\xi, t)-\widehat f(\xi, t)\right |}{|\xi|^{\frac 2{p+1}}} <
{\lambda}. \label{8}
\ee
Since $|\widehat f_\eps(\xi, t)-\widehat
f(\xi, t)| \leq 2$, \fer{8} holds for $\xi \ge R>0$, if $R$ is large
enough (it is enough to let $R>\left(\frac
2{\lambda}\right)^{\frac{p+1}2}$). Therefore, let us prove \fer{8}
when $0<|\xi| \leq R$. Since $\widehat f_\eps$ and $\widehat f$ are
even functions, it is enough to consider $\xi >0$. It
holds
\begin{multline*}
 \partial_t \left(\widehat f_\eps (\xi,t) -\widehat f(\xi,t)\right ) =
 2\int_{0}^{\frac \pi 2} b_\eps(\theta) \left( \widehat f_\eps(\xi \cos^{p+1} \theta ,t)
  \widehat f_\eps(\xi \sin^{p+1} \theta,t) -\widehat f_\eps(\xi,t)\right )\, \d \theta \\
 +2\alpha \xi^{\frac 2{p+1}} \widehat f(\xi, t)+ (p+1) \xi \partial_\xi \widehat f(\xi, t).
 \end{multline*}
 In what follows, in order to shorten  formulas, we will often drop the dependence of the $t$ variable. We get
 \begin{multline*}
  \partial_t \left(\widehat f_\eps (\xi) -\widehat f(\xi)\right ) = -2\alpha \xi^{\frac 2{p+1}}\left(\widehat f_\eps(\xi)- \widehat f(\xi)\right )- (p+1) \xi \partial_\xi \left(\widehat f_\eps(\xi)-\widehat f(\xi)\right ) \\
  +  2\int_{0}^{\frac \pi 2} b_\eps(\theta) \left( \widehat f_\eps(\xi \cos^{p+1} \theta) \widehat f_\eps(\xi \sin^{p+1} \theta) -\widehat f_\eps(\xi)\right )\, \d \theta\\
   + 2\alpha   \xi^{\frac 2{p+1}}\widehat f_\eps(\xi)+ (p+1) \xi \partial_\xi \widehat
   f_\eps(\xi),
 \end{multline*}
that corresponds to
 \be \label{9}
 \partial_t \left( \frac{\widehat f_\eps (\xi) -\widehat f(\xi)}{\xi^{\frac 2{p+1}}}\right) = -2\alpha  \left(\widehat f_\eps(\xi)- \widehat f(\xi)\right )- (p+1)  \frac{\partial_\xi \left(\widehat f_\eps(\xi)-\widehat f(\xi)\right)}{\xi^{\frac {1-p}{p+1}}} +R_\eps(\xi,t)
\ee where
\begin{multline} \label{10}
 R_\eps(\xi, t)=  \frac 2 {\xi^{\frac 2{p+1}}} \int_{0}^{\frac \pi 2} b_\eps(\theta)  \widehat f_\eps(\xi \cos^{p+1} \theta)\left( \widehat f_\eps(\xi \sin^{p+1} \theta) -1\right ) \d \theta
  + 2\alpha \widehat f_\eps(\xi)\\
 +\frac 2 {\xi^{\frac 2{p+1}}} \int_{0}^{\frac \pi 2} b_\eps(\theta) \left( \widehat f_\eps (\xi \cos^{p+1}\theta)-\widehat f_\eps(\xi)\right )\, \d \theta
  + (p+1) \frac {\partial_\xi \widehat f_\eps(\xi) }{\xi^{\frac {1-p}{p+1}}}= A_\eps(\xi, t)+B_\eps(\xi, t).
  \end{multline}
Thanks to assumption \eqref{terza} on $b_\eps$ we obtain
\[
\begin{aligned}
A_\eps (\xi,t) &= \frac 2 {\xi^{\frac 2{p+1}}} \int_{0}^{\frac \pi
2} b_\eps(\theta)  \widehat f_\eps(\xi \cos^{p+1} \theta)\left( \widehat
f_\eps(\xi \sin^{p+1} \theta) -1\right )\, \d \theta
  + 2\alpha \widehat f_\eps(\xi)\\
  &=  2\int_{0}^{\frac \pi 2} b_\eps(\theta)  \sin^{2} \theta\widehat f_\eps(\xi \cos^{p+1} \theta)\frac{ \widehat f_\eps(\xi \sin^{p+1} \theta) -1}  {\left(\xi\sin^{p+1} \theta\right ) ^{\frac 2{p+1}}}\, \d \theta
  + 2\alpha \widehat f_\eps(\xi)\\
  & =  2\int_{0}^{\frac \pi 2} b_\eps(\theta)  \sin^{2} \theta\widehat f_\eps(\xi \cos^{p+1} \theta)\left(\frac{ \widehat f_\eps(\xi \sin^{p+1} \theta) -1}  {\left(\xi\sin^{p+1} \theta\right ) ^{\frac 2{p+1}}}+\alpha \right )\,  \d \theta\\
 &\quad \quad - 2\alpha  \int_{0}^{\frac \pi 2} b_\eps(\theta)  \sin^{2} \theta\left(\widehat f_\eps(\xi \cos^{p+1}
  \theta) -\widehat f_\eps(\xi)\right )\, \d \theta = A_{1,\eps}(\xi,t)+A_{2,\eps}(\xi,t).
\end{aligned}
\]
Let us estimate first $A_{1,\eps}$. Since $\|\widehat
f_\eps\|_\infty \leq 1$, by assumptions \eqref{seconda} and
\eqref{terza} on $b_\eps$ we get for all $t\geq 0$ and $0<\xi \leq
R$
\[
\begin{split}
|A_{1,\eps}(\xi,t)| \leq  2\sup_{t\geq 0, \ c_\eps \leq \theta \leq d_\eps, \ 0<\xi\leq R} \left|\frac{ \widehat f_\eps(\xi \sin^{p+1} \theta,t) -1}  {\left(\xi\sin^{p+1} \theta\right ) ^{\frac 2{p+1}}}+\alpha \right |\\
\leq   2\sup_{ t\geq 0, \ 0<\eta \leq R \left(\sin d_\eps\right
)^{p+1}} \left|\frac{ \widehat f_\eps(\eta,t) -1}  {\eta^{\frac
2{p+1}}}+\alpha \right |
\end{split}
\]
Now,  remembering the uniform convergence in condition \eqref{lim},
we get
\[
\sup_{t\geq 0, \ 0<\xi\leq R}|A_{1,\eps}(\xi,t)|= A(\eps) \to 0,
\quad \eps \to 0.
\]
Let us come to $A_{2,\eps}$. By Lagrange theorem, for all $t\geq 0$
and $0<\xi \leq R$ we obtain
\[
|A_{2,\eps}(\xi, t)| \leq 2\alpha   \int_{0}^{\frac \pi 2}
b_\eps(\theta)  \sin^{2} \theta \left|
\partial_\xi \widehat f_\eps (\bar \xi(\eps, \xi, t))\right | \xi \left|  \cos^{p+1} \theta-1 \right |\, \d \theta
\]
where $0< \xi \cos^{p+1} \theta < \bar \xi(\eps, \xi, t) <\xi \leq
R$.

Thanks to  \eqref{ip-mag-der-sol}, for $\eps$ small enough
 \[
 \sup_{t\geq 0,\ 0<\xi\leq R} \left|\partial_\xi \widehat f_\eps(\xi,t)\right | \leq R^{\frac {1-p}{p+1}} K_1(R) := C'(R)
 \]
and  by assumption \eqref{seconda} on the support of $b_\eps$
$\left|  \cos^{p+1} \theta-1 \right | \leq 1-(\cos d_\eps)^{p+1} \to
0$ for $\eps \to 0$. Finally, owing to assumption \eqref{terza} on
$b_\eps$ we obtain
\[
\sup_{t\geq 0, \ 0<\xi\leq R}|A_{2,\eps}(\xi,t)| \leq 2\alpha C'(R)
R \left( 1-(\cos d_\eps)^{p+1}\right )= B(\eps) \to 0, \quad \eps
\to 0.
\]
Let us now pass to estimate the term $B_\eps$ in \eqref{10}. By
assumption \eqref{terza} on $b_\eps$ we get
\[
\begin{aligned}
B_\eps (\xi, t) &=  2 \int_{0}^{\frac \pi 2} b_\eps(\theta) \frac
{\widehat f_\eps (\xi \cos^{p+1}\theta)-\widehat f_\eps(\xi)} {\xi^{\frac
2{p+1}}}\, \d \theta
  + (p+1) \frac {\partial_\xi \widehat f_\eps(\xi) }{\xi^{\frac {1-p}{p+1}}}\\
  &= - 2 \int_{0}^{\frac \pi 2} b_\eps(\theta) \sin^{2}\theta\, \left( \frac{ \widehat f_\eps (\xi \cos^{p+1}\theta)-\widehat f_\eps(\xi)}  {\left(\xi \cos^{p+1} \theta\right )^{\frac 2{p+1}}- \xi^{\frac 2{p+1}}}
  - \frac {p+1}{2}  \frac {\partial_\xi \widehat f_\eps(\xi) }{\xi^{\frac {1-p}{p+1}}}\right )\, \d \theta.
  \end{aligned}
  \]
By Cauchy  theorem,  for all $t\geq 0 $ and $0<\xi\leq R$ we obtain
\[
|B_\eps(\xi, t)| \leq (p+1) \int_{0}^{\frac \pi 2} b_\eps(\theta)
\sin^{2}\theta\, \left| \frac {\partial_\xi \widehat f_\eps(\tilde \xi)
} {\tilde \xi^{\frac {1-p}{p+1}} } -   \frac {\partial_\xi \widehat
f_\eps(\xi) } {\xi^{\frac {1-p}{p+1}}}
  \right |\, \d \theta,
\]
where $\tilde \xi$  depends on $\eps,\theta, \xi, t$ and lies in
$(\xi \cos^{p+1} \theta, \xi) \subset (0,R)$. Thanks to the uniform
H\"older continuity \eqref{holder}
\[
 (p+1) \int_{0}^{\frac \pi 2} b_\eps(\theta)
 \sin^{2}\theta\,
\left| \frac {\partial_\xi \widehat f_\eps(\tilde \xi) } {\tilde
\xi^{\frac {1-p}{p+1}} } -   \frac {\partial_\xi \widehat f_\eps(\xi) }
{\xi^{\frac {1-p}{p+1}}}
  \right |\, \d \theta
  \leq
  (p+1) \tilde K(R) \int_{0}^{\frac \pi 2} b_\eps(\theta) \sin^{2}\theta\, |\tilde \xi -\xi|^\delta\, \d
  \theta,
 \]
and by condition \eqref{seconda},
 $ |\tilde \xi -\xi|^\delta  < |\xi|^\delta (1-\cos^{p+1} \theta)^\delta \leq R^\delta \left(1-(\cos d_\eps)^{p+1}\right )^\delta \to 0$ as $\eps \to 0$.
 Therefore
 \[
 \sup_{t\geq 0, \, 0<\xi\leq R} |B_\eps(\xi,t) | \leq (p+1)\tilde K(R)  R^\delta  \left(1-(\cos d_\eps)^{p+1}\right )^\delta
 =C(\eps) \to 0, \quad \eps \to 0.
 \]
Going back to equation \eqref{9}, we proved that
\be \label{15}
\sup_{t\geq 0,\, 0<\xi\leq R} |R_\eps(\xi, t)| =
A(\eps)+B(\eps)+C(\eps) \to 0, \quad \eps \to 0.
\ee
Moreover, since
\[
\partial_\xi\left(
 \frac{\widehat f_\eps (\xi) -\widehat f(\xi)}{\xi^{\frac 2{p+1}}}
 \right )= \frac {\partial_\xi\left(\widehat f_\eps (\xi) -\widehat f(\xi)\right)}{\xi^{\frac 2{p+1}}} - \frac 2{p+1} \frac{\widehat f_\eps (\xi) -\widehat f(\xi)}{\xi^{\frac 2{p+1}+1}}
 \]
 we obtain
 \[
  \frac {\partial_\xi\left(\widehat f_\eps (\xi) -\widehat f(\xi)\right)}{\xi^{\frac{1-p}{p+1}}} =
 \xi \partial_\xi\left(
 \frac{\widehat f_\eps (\xi) -\widehat f(\xi)}{\xi^{\frac 2{p+1}}}
 \right )
 + \frac 2{p+1} \frac{\widehat f_\eps (\xi) -\widehat f(\xi)}{\xi^{\frac 2{p+1}}}
 \]
 so that \eqref{9} can be written as
 \begin{multline*}
  \partial_t  \left( \frac{\widehat f_\eps (\xi) -\widehat f(\xi)}{\xi^{\frac 2{p+1}}} \right )= -2\alpha\,  \xi^{\frac 2{p+1}} \, \left(\frac {\widehat f_\eps(\xi)- \widehat f(\xi)}{\xi^{\frac 2{p+1}}}\right )-
   (p+1)\, \xi \, \partial_\xi\left(
 \frac{\widehat f_\eps (\xi) -\widehat f(\xi)}{\xi^{\frac 2{p+1}}}
 \right )\\
 -  2 \, \left(\frac{\widehat f_\eps (\xi) -\widehat f(\xi)}{\xi^{\frac 2{p+1}}}\right )
+R_\eps(\xi,t).
\end{multline*}
Let us denote
\[
y_\eps(\xi, t) = \frac{\widehat f_\eps (\xi, t) -\widehat
f(\xi,t)}{\xi^{\frac 2{p+1}}}.
\]
With this definition we obtained \be\label{13}
\partial_t y_\eps (\xi, t)+ (p+1)\, \xi\, \partial_\xi y_\eps(\xi,t) = -2\left( \alpha \xi^{\frac 2{p+1}}+1\right ) y_\eps (\xi, t) +R_\eps(\xi, t).
\ee Let us further define
\[
z_\eps(\xi,t) := y_\eps (\xi e^{(p+1)t}, t).
\]
Then
\[
\partial_t z_\eps (\xi, t)= (p+1) \xi e^{(p+1)t}\partial_\xi y_\eps (\xi e^{(p+1)t}, t) + \partial_t y_\eps (\xi e^{(p+1)t},
t),
\]
and $z_\eps$ satisfies
\[
\partial_t z_\eps (\xi, t) = -2\left( \alpha \xi^{\frac 2{p+1}}e^{2t} +1\right ) z_\eps (\xi, t) +R_\eps(\xi e^{(p+1)t}, t).
\]
Integrating in time we get
\begin{multline}
\label{14}
z_\eps( \xi, t)= z_\eps (\xi, 0) e^{-\left(\alpha \xi^{\frac 2{p+1}}\left(e^{2t} -1\right ) +2t \right )} \\
+ e^{-\left(\alpha \xi^{\frac 2{p+1}}\left(e^{2t} -1\right ) +2t
\right )} \int_0^t R_\eps (\xi e^{(p+1)s},s) e^{\left(\alpha
\xi^{\frac 2{p+1}}\left(e^{2s} -1\right ) +2s \right )} \, \d s.
\end{multline}
Since $\widehat f_\eps(t)$ and $\widehat f(t)$ correspond to the
same initial data, $z_\eps (\xi, 0)= y_\eps (\xi, 0)= 0$. From
\eqref{14} we obtain
\[
y_\eps(\xi, t)= e^{-\left(\alpha \xi^{\frac
2{p+1}}\left(1-e^{-2t}\right ) +2t \right )}
\int_0^t R_\eps (\xi e^{-(p+1)(t-s)},s) e^{\left(\alpha \xi^{\frac 2{p+1}}e^{-2t} \left(e^{2s}-1\right ) +2s \right )} \, \d s\, . \\
\]
Hence, for all $t\geq 0$ and $0<\xi \leq R$
\[
\begin{aligned}
|y_\eps(\xi, t)|& \leq e^{-\left(\alpha \xi^{\frac
2{p+1}}\left(1-e^{-2t}\right ) +2t \right )}
\int_0^t |R_\eps (\xi e^{-(p+1)(t-s)},s) | e^{\left(\alpha \xi^{\frac 2{p+1}}e^{-2t} \left(e^{2t}-1\right ) +2s \right )} \, \d s\\
&\leq \sup_{t\geq 0, 0<\eta \leq R}|R_\eps (\eta, t)| e^{-2t} \int_0^t e^{2s}\,  \d s  = \left(A(\eps)+B(\eps)+C(\eps) \right ) \frac {e^{-2t}}2\left(e^{2t}-1\right ) \\
&\leq \frac 12  \left(A(\eps)+B(\eps)+C(\eps) \right ).
 \end{aligned}
\]
Finally,  for all ${\lambda} >0$ we proved that there is $\bar \eps
=\bar\eps ({\lambda})$ such that for all $0<\eps <\bar \eps$
\[
\sup_{t\geq 0,\, 0<\xi\leq R} \frac{\left|\widehat f_\eps (\xi, t)
-\widehat f(\xi,t)\right |}{\xi^{\frac 2{p+1}}} <  {\lambda}\, .
\]
This ends the proof.

\hfill $\square$

\bigskip
 In order to give a simple example of initial data which fulfill the assumptions made in Theorem \ref{teo1}, it is enough to consider
 \[
  f_0(v)= \frac 12\left(\widetilde M_p + \phi \right )(v)
 \]
 where $\widetilde M_p(v) = {\cal F}^{-1}\left(e^{-2 \alpha |\xi|^{\frac 2{p+1}}}\right )$ and $\phi$ is an even probability density
 which satisfies $\int |v|^{\frac 2{p+1} +\delta} \phi (v)\, \d v <+\infty$,\ for $0<\delta\leq \frac {2p}{p+1}
 $. Note that an even probability density with bounded energy satisfies this
assumption. In this case $ \lim_{\xi \to 0^+} \frac {1-\widehat
f_0(\xi)}{\xi^{\frac 2{p+1}}}=\alpha$. Moreover $\frac {\partial_\xi
\widehat {\widetilde M}_p(\xi)}{\xi^{\frac{1-p}{p+1}}}$ is uniformly
$\delta$-H\"older continuous on bounded subsets of $(0,+\infty)$. On
the other hand,  thanks to the moment condition we can show that for
$\frac {\partial_\xi \widehat \phi(\xi)}{\xi^{\frac {1-p}{p+1}}}$ is
 uniformly $\delta$--H\"older continuous on $(0,+\infty)$. To this extent, since $\phi $ is an even function we can write
\[
\frac {\partial_\xi \widehat \phi(\xi)}{\xi^{\frac {1-p}{p+1}}}= - \int v\phi (v)\, \frac {\sin (v\xi)}{\xi^{\frac {1-p}{p+1}}}\, \d v.
\]
We remark that, for $x>0$, the function $\frac {\sin x}{x ^{\frac
{1-p}{p+1}}}$ is uniformly $\delta$--H\"older continuous for any
$0<\delta \leq \frac {2p}{p+1}$. In fact for $|x-y| \geq 1$
\[
 \left| \frac {\sin x }{x ^{\frac {1-p}{p+1}}} - \frac {\sin y}{y ^{\frac {1-p}{p+1}}}\right | \leq C \leq C|x-y|^{\delta}.
\]
For $|x-y|\leq 1$ (and $0<x<y$ with no loss in generality), Cauchy
theorem guarantees that there is $x<\bar x<y$ such that
\[
\begin{aligned}
 &\left|
 \frac {\frac {\sin x}{x ^{\frac {1-p}{p+1}}} - \frac {\sin y}{y ^{\frac {1-p}{p+1}}}}{x^{\frac {2p}{p+1}}-y^{\frac {2p}{p+1}}}
 \right| =
 \left|\frac {
 \frac {\cos \bar x}{\bar x^{\frac{1-p}{p+1}}}-\frac {1-p}{p+1}\frac {\sin \bar x}{\bar x^{\frac 2{p+1}}}
 }{\frac {2p}{p+1} \bar x^{\frac {p-1}{1+p}}}\right |\\
& =\frac {p+1}{2p}
 \left|
{\cos \bar x}-\frac {1-p}{p+1}\frac {\sin \bar x}{\bar x}
 \right | \leq C.
 \end{aligned}
\]
Therefore, for $0<\delta \leq \frac {2p}{p+1}$
\[
\begin{aligned}
\left|\frac {\sin x}{x ^{\frac {1-p}{p+1}}} - \frac {\sin y}{y ^{\frac {1-p}{p+1}}}\right |  \leq C \left | x^{\frac {2p}{p+1}}-y^{\frac {2p}{p+1}}\right |
\leq C |x-y|^{\frac {2p}{p+1}}\leq C |x-y|^{\delta}.
\end{aligned}
\]
Finally, for $\xi, \tau  >0 $  we get
 \[
 \begin{aligned}
& \left|\frac {\partial_\xi \widehat \phi(\xi)}{\xi^{\frac {1-p}{p+1}}}-\frac {\partial_\tau \widehat \phi(\tau)}{\tau^{\frac {1-p}{p+1}}}\right | \leq
\int
|v| \phi (v)\,  \left| \frac {\sin(v\xi)}{\xi^{\frac {1-p}{p+1}}} - \frac {\sin (v\tau)}{\tau^{\frac {1-p}{p+1}}} \right |\, \d v\\
& =
\int
|v|^{\frac 2{p+1}} \phi (v)\,  \left| \frac {\sin (v\xi)}{(|v|\xi)^{\frac {1-p}{p+1}}} - \frac {\sin (v\tau)}{(|v|\tau)^{\frac {1-p}{p+1}}} \right |\, \d v\\
& \leq C
 \int  |v|^{\frac 2{p+1}} \phi (v)\, |  |v|\xi-|v|\tau|^\delta \, \d v\\
 &= C |\xi-\tau|^{\delta} \int  |v|^{\frac 2{p+1}+\delta } \phi (v)\, \d v = C |\xi-\tau|^{\delta}.
\end{aligned}
 \]
A second example is furnished by the initial datum
\[
   f_0(v)= \left(M_p \ast \phi\right) (v)
\]
where $M_p(v)= {\cal F}^{-1}\left(e^{- \alpha |\xi|^{\frac 2{p+1}}}\right )$ and $\phi$ is as above.

\bigskip

It is interesting to remark that, under the assumptions on the
initial data given in Theorem \ref{teo1}, we obtain the large-time
convergence result of \cite{PT04}.

\begin{prop}\label{convergenza}
Let $0<p\leq 1$ and let $f_\eps(t)$ the unique solution of the
dissipative Kac equation
\[
 \partial_t \widehat f_\eps(\xi, t)= 2\int_{0}^{\frac \pi 2} b_\eps(\theta) \left( \widehat f_\eps(\xi \cos^{p+1} \theta ,t) \widehat f_\eps(\xi \sin^{p+1} \theta,t) -\widehat f_\eps(\xi,t)\right )\, \d \theta
\]
with even initial density $f_0\geq 0$ satisfying the normalization
condition $\int f_0(v)\, \d v=1$  and assumptions \eqref{ip-lim} and
\eqref{ip-holder} of Theorem \ref{teo1}. Let $\widehat M_p(\xi)=
e^{-\alpha|\xi|^{\frac 2{p+1}}}$ the stationary state with $\alpha$
as in condition \eqref{ip-lim}. Then, for any $0<\delta' \leq
\delta$
\[
\lim_{t\to +\infty} \sup_{\xi \neq 0} \frac {|\widehat f_\eps(\xi,
t)-\widehat M_p(\xi)|} {|\xi |^{\frac 2{p+1}+\delta'} }=0.
\]
\end{prop}

\begin{proof}
As  already remarked, it is enough to prove
\[
 \sup_{\xi \neq 0} \frac {|\widehat f_0(\xi)-\widehat M_p(\xi)|} {|\xi |^{\frac 2{p+1}+\delta'} } <+\infty
\]
and since $\widehat f_0$ and $\widehat M_p$ are even and bounded functions,
it is enough to consider $0< \xi \leq 1$. The stationary state
$M_p$ satisfies the same conditions \eqref{ip-lim} and
\eqref{ip-holder} as $f_0$ with the same constants. Moreover, since
\[
 \lim_{\xi \to 0^+}\frac {\partial_\xi \widehat M_p(\xi)}{ \xi^{\frac {1-p}{p+1}}} =-\frac{2\alpha}{p+1}
\]
and the same is true for $\widehat f_0$ as we proved in
\eqref{lim-der}, we get
\[
 \lim_{\xi \to 0^+} \frac {\partial_\xi \left( \widehat f_0-\widehat M_p\right )(\xi)}{\xi^{\frac {1-p}{p+1}}}=0.
\]
We can therefore apply condition \eqref{ip-holder} to
$\frac {\partial_\xi \left( \widehat f_0-\widehat M_p\right
)(\xi)}{\xi^{\frac {1-p}{p+1}}}$  and pass to the limit for $\tau \to 0$ in order to get
\be\label{hold-diff}
 \left | \frac {\partial_\xi \left( \widehat f_0-\widehat M_p\right )(\xi)}{\xi^{\frac {1-p}{p+1}}}\right | \leq C |\xi|^\delta
\ee for a suitable $C>0$. Now,  by  Cauchy theorem and H\"older
continuity \eqref{hold-diff}  for $0<\xi \leq 1$ and $0<\delta'\leq
\delta$ we get
\[
\begin{aligned}
 \frac {\left|\widehat f_0(\xi)-\widehat M_p(\xi)\right |} {\xi ^{\frac 2{p+1}+\delta'} } &= \frac 1 {\frac 2 {p+1}+\delta'}
 \frac {\left|\partial_\xi \left( \widehat f_0 -\widehat M_p\right )(\tilde \xi)\right |}{\tilde\xi ^{\frac{1-p}{p+1} +\delta'}}\\
 & \leq \frac C {\frac 2 {p+1}+\delta'}
 \frac {\tilde \xi^\delta}{\tilde\xi^{\delta'}} \leq \bar C
 \end{aligned}
\]
for $0<\tilde \xi <\xi$ and $\bar C>0$ and this ends the proof.
\end{proof}

\section{Initial data with finite energy}

In the previous sections, we considered initial data $f_0$ with
unbounded energy and we proved that, if these initial data belong to
a suitable neighborhood of the stationary state, the corresponding
solutions of the dissipative Kac equation \fer{boltz}--\fer{collint}
converge to the solution of a fractional Fokker--Planck equation
when the collisions become grazing. In what follows, we briefly
discuss the simpler case in which the initial data have bounded
energy.  As outlined in the Introduction, in this case any solution
of \eqref{boltz}--\eqref{collint} converges in large times to a
Dirac delta function concentrated in $v=0$. This cooling behavior is
maintained the grazing collision limit, where the collision operator
reduces to a simple linear drift operator, while the diffusive term
is lost. The limit equation is therefore \be\label{drift}
\partial_t \widehat f (\xi, t)= - (p+1) \xi \partial_\xi \widehat f(\xi, t).
\ee A simple calculation shows that if $\widehat f_0$ is the initial
density, this equation has a unique explicit solution in the Fourier
variable \be\label{sol-drift-four} \widehat f(\xi,t)= \widehat
f_0(\xi e^{-(p+1)t}), \ee or, in the physical variable
\[
 f(v,t) = e^{(p+1)t} f_0\left( e^{(p+1)t} v\right ).
\]
The next proposition deals with the aforementioned
situation.

\begin{prop} \label{teo2}
Assume $\{b_\eps(\theta)\}_{\eps >0} \subseteq  L^1([-\frac \pi
2,\frac \pi 2])$ be a family of collision kernels satisfying
Definition \ref{def1}. Let
 $0<p\leq 1$ and let
 $f_\eps(t) \in C^1([0,+\infty), L^1(\R))$ be the solutions of the dissipative Kac equations
\[
\partial_t \widehat f_\eps(\xi, t)= 2\int_{0}^{\frac \pi 2} b_\eps(\theta) \left( \widehat f_\eps(\xi \cos^{p+1} \theta ,t) \widehat f_\eps(\xi \sin^{p+1} \theta,t) -\widehat f_\eps(\xi,t)\right )\, \d \theta  \\
\]
where the even initial density $f_0\geq 0$ satisfies the
normalization condition $\int f_0(v)\, \d v=1$ and has finite
energy, $\int v^2 f_0(v)\, \d v=1$.

Then,
\[
 \lim_{\eps \to 0}\sup_{t\geq 0,\ \xi \neq 0} \frac{\left|\widehat f_\eps(\xi, t)-\widehat f(\xi, t)\right |}{|\xi|^{2}} =0
\]
where $f (t)$, given by \eqref{sol-drift-four} is the solution of
the drift equation
\[
\partial_t \widehat f (\xi, t)= - (p+1) \xi \partial_\xi \widehat f(\xi, t)
\]
 with the same initial data $f_0$.
\end{prop}

\begin{proof}
As in the proof of Theorem \ref{teo1} we consider only $0< \xi \leq
R$. Since
 \begin{multline*}
 \partial_t \left(\widehat f_\eps (\xi,t) -\widehat f(\xi,t)\right ) =  2\int_{0}^{\frac \pi 2} b_\eps(\theta) \left( \widehat f_\eps(\xi \cos^{p+1} \theta ,t) \widehat f_\eps(\xi \sin^{p+1} \theta,t) -\widehat f_\eps(\xi,t)\right )\, \d \theta \\
 + (p+1) \xi \partial_\xi \widehat f(\xi, t),
 \end{multline*}
we get
 \begin{multline*}
  \partial_t \left(\widehat f_\eps (\xi) -\widehat f(\xi)\right ) = - (p+1) \xi \partial_\xi \left(\widehat f_\eps(\xi)-\widehat f(\xi)\right ) \\
  +  2\int_{0}^{\frac \pi 2} b_\eps(\theta) \left( \widehat f_\eps(\xi \cos^{p+1} \theta) \widehat f_\eps(\xi \sin^{p+1} \theta) -\widehat f_\eps(\xi)\right )\, \d \theta
  + (p+1) \xi \partial_\xi \widehat f_\eps(\xi).
 \end{multline*}
Therefore we have
\[
\partial_t \left( \frac{\widehat f_\eps (\xi) -\widehat f(\xi)}{\xi^{2}}\right) = - (p+1)  \frac{\partial_\xi \left(\widehat f_\eps(\xi)-\widehat f(\xi)\right)}{\xi} +\frac{R_\eps(\xi,t)}{\xi^2}
\]
where
\begin{multline*}
 R_\eps(\xi, t)=  2\int_{0}^{\frac \pi 2} b_\eps(\theta) \left( \widehat f_\eps(\xi \cos^{p+1} \theta) \widehat f_\eps(\xi \sin^{p+1} \theta) -\widehat f_\eps(\xi)\right )\, \d \theta
  + (p+1) \xi \partial_\xi \widehat f_\eps(\xi).
  \end{multline*}
Now, via a Taylor expansion with Lagrange reminder
\[
\begin{aligned}
\widehat f_\eps(\xi \cos^{p+1} \theta) \widehat f_\eps(\xi \sin^{p+1} \theta) -\widehat f_\eps(\xi) = \ &
\partial_\xi \widehat f_\eps (\xi) \xi \left ( \cos^{p+1} \theta -1\right )
+ \\
&\frac{\xi^2}2 \left[ \partial^2_{\xi} \widehat f_\eps (\bar \xi) \left ( \cos^{p+1} \theta -1\right )^2 + \widehat f_\eps (\xi)\,  \partial^2_{\xi} \widehat f_\eps (\tilde \xi) \sin^{2(p+1)}\theta \right]\\
& + \frac {\xi^3}2 \, \partial_\xi \widehat f_\eps (\xi) \, \partial^2_{\xi} \widehat f_\eps (\tilde \xi)\left ( \cos^{p+1} \theta -1\right ) \sin^{2(p+1)}\theta \\
&+ \frac {\xi^4}4\, \partial^2_{\xi} \widehat f_\eps (\bar \xi)\, \partial^2_{\xi} \widehat f_\eps (\tilde \xi)  \sin^{2(p+1)}\theta  \left ( \cos^{p+1} \theta -1\right )^2
\end{aligned}
\]
where $\bar \xi$ and $\tilde \xi$ depend on $\xi$, $t$, $\eps$ and
$\theta$ and $\bar \xi \in (\xi, \xi \cos^{p+1}\theta)$, $\tilde \xi
\in(0,\xi \sin^{p+1} \theta)$. By assumption $(c)$ on $b_\eps$,
\[
\begin{aligned}
R_\eps(\xi,t) &= 2 \xi\,  \partial_\xi \widehat f_\eps (\xi)\int_{0}^{\frac \pi 2} b_\eps(\theta)\, \left[ \left( \cos^{p+1} \theta -1\right)+ \frac {p+1}2\sin^2 \theta\right ]\, \d \theta\\
&+ \xi^2 \int_{0}^{\frac \pi 2} b_\eps(\theta) \, \left[  \partial^2_{\xi} \widehat f_\eps (\bar \xi) \left ( \cos^{p+1} \theta -1\right )^2 + \widehat f_\eps (\xi)\,  \partial^2_{\xi} \widehat f_\eps (\tilde \xi)\, \sin^{2(p+1)}\theta \right]\, \d \theta\\
&+ \xi^3  \partial_\xi \widehat f_\eps (\xi)  \int_{0}^{\frac \pi 2} b_\eps(\theta)\, \partial^2_{\xi} \widehat f_\eps (\tilde \xi)\, \left ( \cos^{p+1} \theta -1\right ) \sin^{2(p+1)}\theta\, \d \theta\\
&+ \frac{\xi^4}2 \int_{0}^{\frac \pi 2} b_\eps(\theta)\, \partial^2_{\xi} \widehat f_\eps (\bar \xi)\, \partial^2_{\xi} \widehat f_\eps (\tilde \xi)  \sin^{2(p+1)}\theta  \left ( \cos^{p+1} \theta -1\right )^2\, \d \theta.
\end{aligned}
\]
Since the mass is conserved,  and the energy decays, from \eqref{en}
we obtain
\[
\int v^2 f_\eps(v,t)\, \d v = \exp\left\{- t\int_{-\pi/2}^{\pi/2}
b_\eps (\theta) \left( 1- |\sin\theta|^{2+2p} - |\cos\theta|^{2+2p}
\right)\, {\d\theta}\right\},
\]
Hence, for any $\xi \in \R$, $t\geq 0$ and $\eps >0$
\[
\begin{aligned}
&\left| \widehat f_\eps (\xi,t)\right | \leq | \widehat f_\eps (0,t)|= \int f_\eps(v,t) \, \d v =1,\\
&\left| \partial^2_{\xi}  \widehat f_\eps (\xi,t)\right | \leq  \int v^2 f_\eps(v,t)\, \d v \leq 1,\\
&\left| \partial_{\xi}  \widehat f_\eps (\xi,t)\right | \leq  \int  |v| f_\eps(v,t)\, \d v \leq  \left(\int f_\eps(v,t) \, \d v \right )
\left( \int v^2 f_\eps(v,t)\, \d v\right )^{\frac 12} \leq 1.
\end{aligned}
\]
Moreover, since  $\partial_{\xi}  \widehat f_\eps (0,t)=0$ for all $t\geq 0$ we have
\[
\partial_\xi \widehat f_\eps (\xi,t) = \partial^2_{\xi}  \widehat f_\eps (\xi_\ast,t) \xi
\]
for $\xi_\ast$ depending on $\xi$, $t$ and $\eps$, which implies ${
\dis \left|\frac {\partial_\xi \widehat f_\eps (\xi,t)}{\xi} \right
|\leq 1}$ for all $\xi \neq 0$, $t\geq 0$ and $\eps >0$. Thus, for
$0<\xi\leq R$, $t\geq 0$ and $\eps >0$ we get
\[
\left| \frac{R_\eps(\xi,t)}{\xi^2} \right | \leq C(R)
\int_{0}^{\frac \pi 2} b_\eps(\theta) g(\theta)  \, \d \theta,
\]
where $C(R)>0$ is suitably chosen and
\begin{multline*}
g(\theta)=
\left| \left( \cos^{p+1} \theta -1\right)+ \frac {p+1}2\sin^2 \theta\right | +\left ( \cos^{p+1} \theta -1\right )^2 + \sin^{2(p+1)}\theta\\
+ \left | \cos^{p+1} \theta -1\right | \sin^{2(p+1)}\theta + \sin^{2(p+1)}\theta  \left ( \cos^{p+1} \theta -1\right )^2.
\end{multline*}
Since $b_\eps$ satisfies Definition \ref{def1} it follows
\[
 \int_{0}^{\frac \pi 2} b_\eps(\theta) g(\theta)\, \d \theta = R_\eps \to 0,\quad \eps \to 0.
 \]
We can end the proof as in Theorem \ref{teo1} and obtain
\[
 \lim_{\eps \to 0}\sup_{t\geq 0,\ \xi \neq 0} \frac{\left|\widehat f_\eps(\xi, t)-\widehat f(\xi, t)\right |}{|\xi|^{2}} =0.
\]
\end{proof}

\begin{rem}
It is possible to adapt the previous proof to the classical conservative Kac equation considered in \cite{to1}.  In this case, we obtain a simpler proof of the convergence (in the previous Fourier based metric) of
bounded energy solutions  to the classical Fokker-Planck equation  \eqref{FP}, when the collisions become grazing.
\end{rem}

 If we consider now the fractional Fokker-Planck equations
\eqref{FPF}, it is interesting to remark that the time behavior of
the solutions corresponding to initial data with
bounded energy is completely different. In fact, for any initial
data $f_0\geq 0$ satisfying the normalization condition $\int
f_0(v)\,\d v =1$ the solution of \eqref{FPF} converges to the
corresponding stationary state $M_p=\exp \left\{- \alpha
|\xi|^{2/(1+p)} \right\}$ and that irrespective of how many finite
moments the initial data possess. This implies in particular that
the energy of the solution becomes immediately infinite,  even if
this energy was bounded at the beginning of the evolution.

\begin{prop} \label{conv-FP}
Let  $0<p\leq 1$, $\alpha >0$ and let $f (t)$ be the solution of  the Fokker--Planck equation
\[
\partial_t \widehat f(\xi, t)= -2\alpha |\xi|^{\frac 2{p+1}} \widehat f(\xi, t)-(p+1) \xi \partial_\xi \widehat f(\xi, t)
\]
 with initial density $f_0\geq 0$, satisfying the normalization condition $\int f_0(v)\, \d v=1$.
Then,
\be\label{lim-FP-generale}
 \lim_{t \to +\infty} \|f(t)-M_p\|_{L^1}=0.
\ee
If moreover $f_0$  satisfies
\be \label{momenti}
\int v\, f_0(v)\, \d v=0,\quad \int |v|^{\lambda} f_0(v)\, \d v<+\infty,\quad \lambda = \frac 2{p+1},
\ee
then
\be\label{lim-FP-mom}
 \lim_{t \to +\infty}\sup_{\xi \neq 0} \frac{\left|\widehat f(\xi, t)-\widehat M_p(\xi)\right |}{|\xi|^{\frac 2{p+1}}} =0.
\ee
\end{prop}

\begin{proof}
Thanks to \fer{solFP} we have to prove that
\[
 \left\|\frac 1{\beta(t)} f_0\left( \frac \cdot {\beta (t)}\right ) \ast \frac 1{\gamma(t)} M_p \left( \frac \cdot {\gamma(t)}\right )- M_p
 \right \|_{L^1}\freccia 0, \quad t \to +\infty.
\]
To simplify notations, we will write
\[
 \begin{aligned}
 & f_{0,\beta} (v) =\frac 1{\beta(t)} f_0\left( \frac v {\beta (t)}\right ) \\
 & M_{p,\gamma} (v) =\frac 1{\gamma(t)} M_p \left( \frac v {\gamma(t)}\right ).
\end{aligned}
\]
Then
\[
 \left\|f_{0, \beta} \ast M_{p,\gamma}- M_p \right \|_{L^1} \leq
 \left\|f_{0, \beta} \ast M_{p,\gamma} - f_{0, \beta} \ast M_p\right \|_{L^1} + \left\| f_{0, \beta} \ast M_p -M_p\right \|_{L^1}.
 \]
Since $\beta(t) \to 0$ for $t\to +\infty$ and $f_0 \in L^1$ with
$\int f_0(v)\d v=1$, it is classical that $\left\| f_{0, \beta} \ast
M_p -M_p\right \|_{L^1}\to 0$ \cite[page 10]{SW}. Moreover, since
\[
\left\|f_{0, \beta} \ast \left( M_{p,\gamma} -M_p\right )\right
\|_{L^1} \leq \left\|f_{0, \beta}\right \|_{L^1} \left
\|M_{p,\gamma} -M_p\right \|_{L^1},
\]
and
\[
\int f_{0,\beta}(w)\, \d w = \int f_0(w)\, \d w=1,
\]
it is enough to prove  that $ a(t)= \left\|  M_{p,\gamma}
-M_{p}\right \|_{L^1} $ vanishes as time goes to infinity. For any
given $R >0$, we have \be\label{spez}
\begin{aligned}
a(t)&= \int \left|\frac 1{\gamma(t)} M_p \left( \frac v {\gamma(t)}\right ) - M_p(v)\right |\, \d v\\
&=\int_{|v| \leq R} \left|\frac 1{\gamma(t)} M_p \left( \frac v {\gamma(t)}\right ) - M_p(v)\right |\, \d v +  \int_{|v| > R} \left|\frac 1{\gamma(t)} M_p \left( \frac v {\gamma(t)}\right ) - M_p(v)\right |\, \d v
\end{aligned}
\ee
Let us consider first the second term.
 Since $\gamma (t) = (1-e^{-2t})^{\frac {p+1}2} \to 1$ for $t\to
 +\infty$, there is $t_0$ so that for $t\geq t_0$ we have $\frac 12 \leq \gamma(t) <1$.
 So,
 for $t\geq t_0$
\[
\begin{aligned}
&\int_{|v| > R} \left|\frac 1{\gamma(t)} M_p \left( \frac v {\gamma(t)}\right ) - M_p(v)\right |\, \d v  \\
&\leq  \int_{|v| > R} \frac 1{\gamma(t)} M_p \left( \frac v {\gamma(t)}\right )\, \d v +
\int_{|v| > R}M_p(v)\, \d v\\
& = \int_{|v| > \frac R{\gamma (t)}}  M_p (v) \, \d v + \int_{|v| > R} M_p(v)\, \d v \leq 2 \int_{|v| > R}M_p(v)\, \d v.
\end{aligned}
\]
Since $M_p\in L^1$, there is $R_1 = R_1(\eps)>0$ so that  for $R\geq R_1$
\[
2\int_{|v| > R}M_p(v)\, \d v  < \frac \eps 2.
\]
Let us come to the first integral in \eqref{spez}.
For any $v\in\R$ we have
\[
\begin{aligned}
& \left| \frac 1{\gamma(t)} M_p \left( \frac v {\gamma(t)}\right ) - M_p(v)\right | =
\left| \int \left(\widehat M_p(\gamma(t) \xi) - \widehat M_p(\xi)\right ) e^{iv\xi}\, \d \xi\right|\\
&\leq  \int \left|\widehat M_p(\gamma(t) \xi) - \widehat M_p(\xi)\right | \d \xi.
\end{aligned}
\]
Since  $\lim_{t\to +\infty} \widehat M_p(\gamma(t) \xi)= \widehat
M_p(\xi)$ for all $\xi$ and for $t\geq t_0$ we have $\frac 12 \leq
\gamma(t) <1$ and so  $\widehat
M_p(\gamma (t)\xi)\leq  \widehat M_p(\xi/2)$, by Lebesgue
theorem there is $t_1= t_1(\eps) >0$ so that for $t\geq \max
(t_0,t_1)$ we have
\[
 \int \left|\widehat M_p(\gamma(t) \xi) - \widehat M_p(\xi)\right | \d \xi \leq \frac \eps {4R}
\]
and so
\[
 \int_{|v| \leq R} \left|\frac 1{\gamma(t)} M_p \left( \frac v {\gamma(t)}\right ) - M_p(v)\right |\, \d v
  \leq \frac \eps 2.
\]
%
Letting $\bar t = \max (t_0,t_1)$,  the proof of \eqref{lim-FP-generale} is completed.

To prove  \eqref{lim-FP-mom}, we remark that conditions
\eqref{momenti} imply
\be\label{ip-magg}
\sup_{ \xi \neq 0}\frac{\left|1-\widehat
f_0(\xi)\right |}{|\xi|^{\frac 2{p+1}}} \leq C
\ee
for $C>0$ suitably chosen.
Now, for $\xi \neq 0$ and $t\geq 0$
\[
\begin{aligned}
 &\frac{\left|\widehat f(\xi, t)-\widehat M_p(\xi)\right |}{|\xi|^{\frac 2{p+1}}} =
 \frac{\left|\widehat f_0\left( \xi e^{-(p+1)t}\right ) e^{-\alpha |\xi|^{\frac 2{p+1}}(1-e^{-2t})}-e^{-\alpha|\xi|^{\frac 2{p+1}}} \right |}
 {|\xi|^{\frac 2{p+1}}}\\
 &\leq
 \left|\widehat f_0\left( \xi e^{-(p+1)t}\right )  \right |
 \frac{ \left|e^{-\alpha |\xi|^{\frac 2{p+1}}(1-e^{-2t})}-e^{-\alpha|\xi|^{\frac 2{p+1}}}\right |}
 {|\xi|^{\frac 2{p+1}}} + e^{-\alpha|\xi|^{\frac 2{p+1}}}
 \frac{\left| \widehat f_0\left( \xi e^{-(p+1)t}\right ) -1\right |}
 {|\xi|^{\frac 2{p+1}}}.
 \end{aligned}
\]
Thanks to condition \eqref{ip-magg} on $f_0$  for $\xi \neq 0$ we
get
\[
\begin{aligned}
 \frac{\left|\widehat f(\xi, t)-\widehat M_p(\xi)\right |}{|\xi|^{\frac 2{p+1}}} &\leq e^{-\alpha |\xi|^{\frac 2{p+1}}(1-e^{-2t})}
 \frac{ 1- e^{-\alpha |\xi|^{\frac 2{p+1}} e^{-2t}}}
 {|\xi|^{\frac 2{p+1}}} + e^{-2t}\frac{1- \widehat f_0\left( \xi e^{-(p+1)t}\right )}
 {|\xi e^{-(p+1)t}|^{\frac 2{p+1}}}\\
 &\leq C (\alpha+1) e^{-2t}  \to 0, \quad t \to +\infty.
 \end{aligned}
\]
\end{proof}

\bigskip \noindent
{\bf Acknowledgment:}  The authors (A.P. and G.T) acknowledge
support by MIUR project ``Optimal mass transportation, geometrical
and functional inequalities with applications''. This paper has been
written within the activities of the National Group of Mathematical
Physics (GNFM) and of the National Group of Mathematical Analysis, Probability and Applications
(GNAMPA) of INDAM.

%

\end{document}